\documentclass[12pt,twoside]{amsart}
\usepackage{amssymb,latexsym,amsthm,amsmath,mathtools,amsfonts,hyperref,fancyhdr,comment}
\usepackage{color}

\hypersetup{
	colorlinks=true,
	linkcolor=blue,
	filecolor=magenta,
	urlcolor=cyan,
	citecolor=red,
}

\long\def\symbolfootnote[#1]#2{\begingroup%
	\def\thefootnote{\fnsymbol{footnote}}\footnote[#1]{#2}\endgroup}

\newcommand{\lt}{\mathfrak l}

\newcommand{\supp}{\rm{Supp}}

\makeatletter
\def\imod#1{\allowbreak\mkern10mu({\operator@font mod}\,\,#1)}
\makeatother

\newtheorem{theorem}{Theorem}[section]
\newtheorem{lemma}[theorem]{Lemma}
\newtheorem{corollary}[theorem]{Corollary}
\newtheorem{proposition}[theorem]{Proposition}
\newtheorem*{theorem*}{Theorem}
\theoremstyle{definition}

\newtheorem{conjecture}[theorem]{Conjecture}
\numberwithin{equation}{section}
\newcommand{\ignore}[1]{}

\newcommand{\mynote}[1]{}
\newenvironment{manualtheorem}[1]{%
	\renewcommand\thetheorem{#1}%
	\theorem
}{\endtheorem}

\begin{document}

	\title{Alternating groups as products of cycle classes}
	
	\author{Harish Kishnani}
	\address{Indian Institute of Science Education and Research Mohali, Knowledge City, Sector 81, Mohali 140306, India}
	\email{harishkishnani11@gmail.com}
	
	\author{Rijubrata Kundu}
	\address{Indian Institute of Science Education and Research Mohali, Knowledge City, Sector 81, Mohali 140306, India}
	\email{rijubrata8@gmail.com}
	
	\author{Sumit Chandra Mishra}
	\address{Indian Institute of Science Education and Research Mohali, Knowledge City, Sector 81, Mohali 140306, India}
	\email{sumitcmishra@gmail.com}

	\subjclass[2010]{20B30, 20D06, 05A05, 20B05}
	\today 
	\keywords{alternating groups, product of conjugacy classes}
	\begin{abstract}
		Given integers $k,l\geq 2$, where either $l$ is odd or $k$ is even, let $n(k,l)$ denote the largest integer $n$ such that each element of $A_n$ is a product of $k$ many $l$-cycles. In 2008, M. Herzog, G. Kaplan and A. Lev  proved that if $k,l$ both are odd, $3\mid l$ and $l>3$, then $n(k,l)=\frac{2}{3}kl$. They further conjectured that if $k$ is even and $3\mid l$, then $n(k,l)=\frac{2}{3}kl+1$. In this article, we prove this conjecture. We also prove that $n(k,3)=2k+1$ if $k$ is odd. 
	\end{abstract}
	
	\maketitle
	
	\section{Introduction}
	
	Let $S_n$ (resp. $A_n$) denote the symmetric group (resp. alternating group) on $n$ letters. Given natural numbers $k,l\geq 2$, let $n(k,l)$ denote the largest natural number $n$ such that every element in $A_n$ can be written as a product of $k$ many $l$-cycles. Note that the number $n(k,l)$ makes sense only when $l$ is odd or $k$ is even. Alternatively, the definition of $n(k,l)$ can be visualized through the notion of product of conjugacy classes in groups. Given a finite group $G$ and a conjugacy class $C$ of $G$, let $C^r:=\{x_1x_2\cdots x_r\mid x_i\in C, 1\leq i\leq r\}$. For a given integer $l$ such that $2\leq l\leq n$, let $C_l$ denote the conjugacy class of a $l$-cycle in $S_n$. Note that $C_l$ consists of all $l$-cycles in $S_n$. Then, $n(k,l)$ is the largest integer $n$ such that $C_l^k=A_n$. 
	
	E. Bertram in \cite{be} determined the value of $n(k,l)$ for $k=2$. He proved the following result.
	
	\begin{theorem}[Corollary 2.1, \cite{be}]\label{bounds_n(2,l)} 
		Given natural numbers $n,l\geq 2$,  every element of $A_n$ can be written as a product of two $l$-cycles in $S_n$ if and only if $\lfloor\frac{3n}{4}\rfloor\leq l\leq n$ or $n=4$ and $l=2$. 
	\end{theorem}
	
	Note that in terms of conjugacy classes, the above result simply states that $C_l^2=A_n$ if and only if $\lfloor\frac{3n}{4}\rfloor\leq l\leq n$ or $n=4$ and $l=2$. From the above result the following corollary can be easily derived.
	
	\begin{corollary}\label{n(2,l)_exact_value}
		For $l>2$, we have $n(2,l)=\lfloor\frac{4l}{3}\rfloor+1$.
	\end{corollary} 

	In \cite{bh}, the authors have dealt with the case $n(3,l)$ and $n(4,l)$. 
	
	\begin{theorem}[Theorem 2, \cite{bh}]\label{n(3,l)_bounds}
		For $n,l\geq 2$, each element $\sigma\in A_n$ is a product of three $l$-cycles in $S_n$ if and only if $l$ is odd and either $\lceil\frac{n}{2}\rceil\leq l \leq n$, or $n=7$ and $l=3$. 
	\end{theorem}
	
	\noindent Once again the following corollary is easy to deduce.
	
	\begin{corollary}
		If $l>3$ and $l$ is odd then $n(3,l)=2l$.
	\end{corollary}

	\begin{theorem}[Theorem 3, \cite{bh}]\label{n(4,l)_bounds}
		For $n,l\geq 2$, each element $\sigma\in A_n$ is a product of four $l$-cycles in $S_n$ if and only if:
		\begin{enumerate}
			\item $\lceil \frac{3n}{8} \rceil \leq l\leq n$ if $n\not\equiv 1\;(\mathrm{mod}\;8)$,
			\item $\lfloor\frac{3n}{8}\rfloor \leq l\leq n$ if $n\equiv 1\;(\mathrm{mod}\;8)$,
			\item $n=6$ and $l=2$.
		\end{enumerate}
	\end{theorem}
	
	\begin{corollary}\label{n(3,l)_n(4,l)_exact_value}
		Let $l>2$ be a positive integer and $3\mid l$. Then $n(4,l)=\frac{8l}{3}+1$.
	\end{corollary}
	
	In an attempt to generalize these results, the authors in \cite{hgl} gave a general upper bound for $n(k,l)$ (see Theorem~\ref{general_upper_bound_n(k,l)}). Further, they conjectured the following:
	
	\begin{conjecture}[Conjecture 1,1, \cite{hgl}]
		Let $k,l\geq 2$ and assume that $l$ is odd or $k$ is even. Then $\lfloor \frac{2}{3}kl \rfloor\leq n(k,l)\leq \lfloor\frac{2}{3}kl\rfloor+1$.
	\end{conjecture}
	
	\noindent They verified the above conjecture when $3\mid l$, $l\geq 9$ and $l$ is odd. Further, the exact value of $n(k,l)$ when both $k,l$ are odd, $3\mid l$ and $l\geq 9$ was obtained.
	\begin{theorem}[Theorem 3.4, \cite{hgl}]\label{upper_lower_bound_for_l_divisible_by_3_l_odd}
		Let $k\geq 2$ and $3\mid l$, $l\geq 9$ and $l$ is odd. Then $\frac{2}{3}kl\leq n(k,l)\leq \frac{2}{3}kl+1$. Furthermore, if $k$ is odd then $n(k,l)=\frac{2}{3}kl$.
	\end{theorem}
	
	In the light of the above result, they also conjectured the following:
	
	\begin{conjecture}[Conjecture 1.2, \cite{hgl}]\label{mainconjecture}
		Let $k,l$ be positive integers and assume that $k$ is even and $3\mid l$. Then $n(k,l)=\frac{2}{3}kl+1$. 
	\end{conjecture}
	
	In this article, we prove the above conjecture in the affirmative. The following two theorems are the main results of this paper.
	
	\begin{manualtheorem}{A}\label{n(k,3)_exact_value}
		Suppose $k\geq 2$ be any natural number. Then, $n(k,3)=2k+1$.
	\end{manualtheorem}

	\begin{manualtheorem}{B}\label{n(k,l)_exact_value_k_even}
		Let $k\geq 2$ and $l>3$ be two natural numbers such that $k$ is even and $3\mid l$. Then $n(k,l)=\frac{2}{3}kl+1$.
	\end{manualtheorem}

	Note that the value of $n(k,3)$ is same whether $k$ is even or odd, unlike the case when $l$ is odd and $l>3$ where $n(k,l)=\frac{2}{3}kl$ for odd $k$.
	
	\medskip
	
	As has been discussed in the beginning, this problem is closely related to the product of conjugacy classes in finite groups in which extensive literature is available. It is well known that if $G$ is a non-abelian finite simple group and $C$ is a non-trivial conjugacy class of $G$, then there exists a positive integer $n\geq 2$ such that $C^n=G$. Conversely, let $G$ be a finite group and suppose that for every non-trivial conjugacy class $C$ of $G$, there exists a natural number $n\;(\geq 2)$ such that $C^n=G$. Then $G$ is a non-abelian finite simple group. Given a non-abelian finite simple group $G$, the least positive integer $n$ such that $C^n=G$ for every non-trivial conjugacy class $C$ of $G$ is called the covering number of the group $G$. The covering numbers of many simple groups (including the alternating groups) have been studied extensively (see \cite{ash} and the references therein). We also refer the readers to a recent survey article \cite{bfm} for more on product of conjugacy classes in finite groups.
	
	\medskip
	
	It is also worthwhile to mention the well-known Thompson's conjecture which states that for every non-abelian finite simple group $G$, there exists a conjugacy class $C$ such that $C^2=G$. Note that Theorem~\ref{bounds_n(2,l)}  confirms the conjecture for the alternating groups.  Thompson's conjecture has been  proved for a large number of finite simple groups, although it remains open for infinite family of groups. We also mention the Ore's conjecture which states that every element in a non-abelian simple group is a commutator. Note that the Thompson's conjecture implies the Ore's conjecture. The Ore's conjecture has now been solved completely. We refer the readers to a beautiful article by G. Malle (see \cite{ma}) for more on this.
	

	\subsection{Notations} We use some standard notations throughout this paper. Let $\sigma \in S_n$. If $\sigma$ is a cycle then $\mathfrak{l}(\sigma)$ denotes the length of $\sigma$. Let $\supp(\sigma):=\{i \mid \sigma(i)\neq i\}$ denote the support of $\sigma$. Two permutations $\sigma,\tau \in S_n$ are called disjoint if $\supp(\sigma)\cap \supp(\tau) = \emptyset$. It is well known that any permutation can be written as a product of disjoint cycles. Moreover, such a decomposition is unique up to cyclic shifts within each cycle and the order in which the cycles are written. For $\sigma \in S_n\setminus\{1\}$, $dcd*(\sigma)$ denotes  a disjoint cycle decomposition of $\sigma$ where each cycle has length greater than 1. Let $m_{\sigma}$ denote the number of symbols moved by $\sigma$, that is, $m_{\sigma}=|\supp(\sigma)|$. Let $n_{\sigma}$ denote the number of cycles in $dcd*(\sigma)$. For any $i\geq 2$, let $n_i(\sigma)$ denote the number of $i$-cycles in $dcd*(\sigma)$. It is clearly seen that $\sum_{i}n_i=n_{\sigma}$ and $\sum_{i}in_i=m_{\sigma}$. Finally, we mention that product of permutations will be executed from right to left.
	
	
	\section{Even permutations as product of three cycles}
	In this section we prove Theorem~\ref{n(k,3)_exact_value}. The following result which gives general upper bound for $n(k,l)$ will be used throughout this article.
	
	\begin{theorem}[Theorem 3.3, \cite{hgl}]\label{general_upper_bound_n(k,l)}
		Let $k,l$ be natural numbers such that $k\geq 2$ and $l>2$. Suppose that either $l$ is odd or $k$ is even. Denote $n_1=\lfloor\frac{2kl}{3}\rfloor$ and $\delta=\frac{2kl}{3}-n_1$. Then
		\begin{enumerate}
			\item If $n_1\equiv 3\;(\rm{mod}\;4)$, then $n(k,l)\leq n_1$.
			\item If $n_1\equiv 1\;(\rm{mod}\;4)$, then $n(k,l)\leq n_1+1$.
			\item If $n_1\equiv 2\;(\rm{mod}\;4)$, then $n(k,l)\leq n_1+1$; if we further assume that $l>3$ and $\delta\in \{0,\frac{1}{3}\}$, then $n(k,l)\leq n_1$;
			\item If $n_1\equiv 0\;(\rm{mod}\;4)$, then $n(k,l)\leq n_1+1$.
		\end{enumerate}
	\end{theorem}
	
	Before moving on to the proof of Theorem~\ref{n(k,3)_exact_value}, we mention two lemmas (which can also be found in \cite{hgl}) that will be used throughout. The proofs are left as they are easy.
	
	\begin{lemma}\label{even_permutation_characterization}
		$\sigma\in A_n$ if and only if $m_{\sigma}+n_{\sigma}$ is even.
	\end{lemma}
	
	\noindent \textbf{Notation:} Let $P(k,l;n)$ denote the set of all $\sigma \in A_n$ such that $\sigma$ can be written as a product of $k$ many $l$-cycles in $A_n$.
	
	\begin{lemma}\label{increasing_cycles}
		Let $k,l\in \mathbb{N}$ and $l\geq 2$. If $l$ is odd then $P(k,l;n)\subseteq P(k+1,l;n)$. If $l$ is even then $P(k,l;n)\subseteq P(k+2,l;n)$.
	\end{lemma}

	\begin{proof}[\textbf{Proof of Theorem~\ref{n(k,3)_exact_value}}]
		By Theorem~\ref{general_upper_bound_n(k,l)}, it is enough to show that any $\sigma \in A_{2k+1}$ can be written as a product of $k$ many 3-cycles. We prove this by induction on $k$. 
		
		\medskip
		
		\noindent Base case: For $k=2$, the result follows from Corollary~\ref{n(2,l)_exact_value}.
		
		\medskip
		
		\noindent Induction hypothesis: Let us assume that the result holds for $2\leq j\leq k$, that is, $n(j,3)=2j+1$ for every $2\leq j\leq k$.
		
		\noindent We show that every element in $A_{2k+3}$ can be written as a product of $k+1$ many 3-cycles. 
		
		\medskip
		
		\noindent \textbf{\underline{Case I}:} Suppose $\sigma\in A_{2k+3}$ is a cycle of length $2k+3$. Thus $\sigma=(a_1\;a_2\;\cdots\; a_{2k+3})$. Suppose $k$ is odd. We can write $\sigma=\sigma_1\sigma_2$ where
		$$\sigma_1=(a_1\;\cdots\;a_{k+2}), \text{and}\;\sigma_2=(a_{k+2}\;\cdots\;a_{2k+3}).$$
		Note that $\sigma_1,\sigma_2\in A_{k+2}$ and $k+2<2k+3$. Using induction hypothesis, both $\sigma_1,\sigma_2\in P(\frac{k+1}{2},3;k+2)$. Thus $\sigma\in P(k+1,3;2k+3)$.
		Suppose $k$ is even. We can write $\sigma=\sigma_1\sigma_2$ where
		$$\sigma_1=(a_1\;\cdots\;a_{k+1}), \text{and}\;\sigma_2=(a_{k+1}\;\cdots\;a_{2k+3}).$$
		Note that $\sigma_1\in A_{k+1}$ and $\sigma_2\in A_{k+3}$. Further, $k+1<2k+3$ and $k+3<2k+3$. Using induction hypothesis, $\sigma_1 \in P(\frac{k}{2},3;k+1)$ and $\sigma_2\in P(\frac{k+2}{2},3;k+3)$. Thus $\sigma \in P(k+1,3;2k+3)$ and we are done. 
		
		\medskip
		
		\noindent \textbf{\underline{Case II}:} Assume that there exists a cycle $\rho$ of odd length in $dcd*(\sigma)$ with $\lt(\rho)<2k+3$. So $\sigma=\rho \tau$ where $\tau$ is the product of disjoint cycles appearing in $dcd*(\sigma)$ different from $\rho$. Clearly $\tau \in A_{2k+4-\lt(\rho)}$ and since  $2k+4-\lt(\rho)<2k+3$, we conclude that $\tau$ can be written as a product of $\frac{2k+4-\lt(\rho)-1}{2}$ many 3-cycles. On the other hand $\rho\in A_{\lt(\rho)}$ and thus $\rho$ can be written as a product of $\frac{\lt(\rho)-1}{2}$ many 3-cycles. We conclude that $\sigma$ can be written as a product of $\left(\frac{2k+4-\lt(\rho)-1}{2}+\frac{\lt(\rho)-1}{2}\right)$ many 3-cycles. This proves that $\sigma\in P(k+1,3;2k+3).$
		\medskip
		
		We can now assume that every cycle in $dcd*(\sigma)$ is of even length. Note that this implies $n_{\sigma}$ is even. By Lemma~\ref{even_permutation_characterization}, we conclude that $m_{\sigma}$ is also even. Now we make two more cases.
		
		\noindent \textbf{\underline{Case III}:} Suppose that $m_{\sigma}\leq 2k$. We write $r=2k+3-m_{\sigma}$. Then $r\geq 3$ and $r$ is odd. Clearly $\sigma \in A_{2k+4-r}$. Since $2k+4-r<2k+3$, $\sigma$ can be written as a product of $\frac{2k+4-r-1}{2}=\frac{2k+3-r}{2}$ many 3-cycles. In other words, $\sigma\in P(\frac{2k+3-r}{2},3;2k+3)$. Since $\frac{2k+3-r}{2}<k+1$, using Lemma~\ref{increasing_cycles} we conclude that $\sigma\in P(k+1,3;2k+3)$.
	
		\medskip
		
		\noindent \textbf{\underline{Case-IV}:} The last remaining case is when $m_{\sigma}=2k+2$. Let $s=n_{\sigma}$. We write $\sigma=\sigma_1\cdots \sigma_s$ where $\sigma_i$'s are the distinct cycles appearing in $dcd*(\sigma)$. The product of two even length cycles can be written as follows:
		$$(a_1\;\ldots\;a_{2d})(b_1\;\ldots\; b_{2e})=(a_1\;a_3\;a_4\ldots\;a_{2d})(a_1\;b_1\;a_2)(b_1\;b_{2e}\;a_1)(b_1\;\ldots\; b_{2e-1}).$$
		
		For $d=1$ (resp. $e=1$), we define the first (resp. last) term of the RHS in the above equation to be identity.
		Using the above, for each $1\leq i\leq \frac{s}{2}$, we can write $$\sigma_{2i-1}\sigma_{2i}=\sigma_{2i-1,1}\sigma_{2i-1,2}\sigma_{2i,2}\sigma_{2i,1},$$
		where $\sigma_{2i-1,1}$ and $\sigma_{2i,1}$ are cycles of length $\lt(\sigma_{2i-1})-1$ and $\lt(\sigma_{2i})-1$ respectively with both $\sigma_{2i-1,2}$ and $\sigma_{2i,2}$ being 3-cycles.  It is also clear that for each $1\leq i\leq \frac{s}{2}$ both $\lt(\sigma_{2i-1})-1, \lt(\sigma_{2i})-1$  are odd and they are less than $2k+3$. Thus for each such $i$ we conclude that $\sigma_{2i-1,1}$ can be written as a product of $\frac{\lt(\sigma_{2i-1})-2}{2}$ many 3-cycles and $\sigma_{2i,1}$ can be written as a product of $\frac{\lt(\sigma_{2i})-2}{2}$ many 3-cycles. Since
		$$\sigma=\sigma_1\cdots \sigma_s =\prod\limits_{i=i}^{\frac{s}{2}}\sigma_{2i-1}\sigma_{2i}=\prod\limits_{i=i}^{\frac{s}{2}}\sigma_{2i-1,1}\sigma_{2i-1,2}\sigma_{2i,2}\sigma_{2i,1},$$
		we conclude that $\sigma$ can be written as a product of $s+\sum\limits_{i=1}^{\frac{s}{2}}\left(\frac{\lt(\sigma_{2i-1})-2}{2}+\frac{\lt(\sigma_{2i})-2}{2}\right)$ 3-cycles. We have
	
		\begin{eqnarray*}
		s+\sum\limits_{i=1}^{\frac{s}{2}}\left(\frac{\lt(\sigma_{2i-1})-2}{2}+\frac{\lt(\sigma_{2i})-2}{2}\right) &=&\sum_{i=1}^{s}\frac{\lt(\sigma_{i})}{2} = k+1.
		\end{eqnarray*}
		Thus $\sigma$ can be written as a product of $k+1$ many 3-cycles. This completes the proof.
	\end{proof}

	\section{Lower bound for $n(k,l)$ and a key decomposition result}
	
	This section is a build-up to the proof of Theorem~\ref{n(k,l)_exact_value_k_even}. The following result of \cite{hgl}, which played a key role in proving Theorem~\ref{upper_lower_bound_for_l_divisible_by_3_l_odd}, will be required later. 
	
	\begin{theorem}[Theorem 2.6, \cite{hgl}] \label{auxillary_theorem_hgl}
		Let $k,l,n\in \mathbb{N}$ be such that $k\geq 2$ and $l$ is odd and suppose that $9\leq l\leq n\leq \frac{2}{3}kl+1$. Moreover, let $\sigma\in A_n$ be such that $n_{\sigma}\leq \frac{n+2}{3}$ if $k\geq 3$. Then $\sigma \in P(k,l;n)$.
	\end{theorem}

	The following decomposability result is crucial for the proof of the main theorem.
	
	\begin{proposition}\label{main_decomposibility_result}
		Let $l\geq 9$ be a natural number, $3\mid l$ and $n=\frac{2}{3}kl+1$ where $k\geq 6$ is an even number. Let $\sigma \in A_n$ be such that $m_{\sigma}=n=\frac{2}{3}kl+1$ and $n_\sigma>\frac{n+2}{3}$. Then, $\sigma$ can be written as a product of two disjoint permutations $\phi$ and $\tau$ such that $\phi\in A_{\frac{2}{3}(k-2)l+\epsilon}$ and $\tau \in A_{\frac{4l}{3}+1-\epsilon}$ where $\epsilon \in \{0,1\}$.
	\end{proposition}
	
	We prove this by making some cases which we present in the form of lemmas. We mention that in every case the $\tau$ in the above theorem is obtained by taking a suitable product of certain cycles from $dcd*(\sigma)$. Thus it is enough to produce this $\tau$ suitably since the $\phi$ can then be obtained by just taking those cycles in $dcd*(\sigma)$ which do not appear in $dcd*(\tau)$. Recall that for $i\geq 2$, $n_i(\sigma)$ denote the number of $i$-cycles in $dcd*(\sigma)$.
	
	\begin{lemma}\label{Theorem_4.2_for_l=9}
		Assume that the hypothesis of Proposition~\ref{main_decomposibility_result} holds with $l=9$. Then, $\sigma$ $(\in A_{6k+1})$ can be written as a product of two disjoint permutations $\phi$ and $\tau$ such that $\phi\in A_{6k-11}$ and $\tau \in A_{12}$. 
	\end{lemma}
	
	\begin{proof}
		Note that $n=6k+1=m_{\sigma}$ where $k\geq 6$ and $n_{\sigma}>2k+1$. For $i\geq 2$, we write $n_i$ to mean $n_i(\sigma)$. If $n_2\geq 6$, then we choose $\tau$ to be a product of six distinct $2$-cycles from $dcd*(\sigma)$ and we are done. Suppose now that $n_2<6$. We have
		\begin{eqnarray*}
		2n_2 + 3(n_{\sigma}-n_2)\leq n &\implies&  n_{\sigma}\leq \frac{n+n_2}{3}<\frac{6k+7}{3}=2k+2+\frac{1}{3}
		\end{eqnarray*} 
		Thus, when $n_2<6$, the only possibility for $n_{\sigma}$ is $2k+2$. We see that $m_{\sigma}+n_{\sigma}$ is odd in this case which implies that $\sigma \in S_n\setminus A_n$ (by Lemma~\ref{even_permutation_characterization}), a contradiction. The result now follows.
	\end{proof}

	\begin{lemma}\label{Theorem_4.2_for_k>6}
		Proposition~\ref{main_decomposibility_result} holds for $l\geq 12$ and $k>6$.
	\end{lemma}

	\begin{proof}
	 	Let us write $l=3m$ where $m\geq 4$ and $k=3N+r$ where $N\geq 2$ and $r\in \{0,1,2\}$. Thus $n=\frac{2}{3}kl+1=6mN+2rm+1$ and $n_{\sigma}>2mN+\frac{2}{3}rm+1$. Once again we simply write $n_i$ to mean $n_i(\sigma)$. If $n_2\geq \frac{2l}{3}=2m$, then we can choose $\tau$ to be a product of distinct $\frac{2l}{3}$ many $2$-cycles from $dcd*(\sigma$). Clearly $\tau \in A_{\frac{4l}{3}}$ and we are done. So we can now assume that $n_2<\frac{2l}{3}=2m$.
	 
	 	We now derive a lower bound for $n_2$. We have
	 	\begin{eqnarray*}
	 	2n_2 + 3(n_{\sigma}-n_2)\leq n &\implies&  n_{2}\geq 3n_{\sigma}-n> 2
	 	\end{eqnarray*} 
	 	The last inequality follows since $n_{\sigma} > 2mN + \frac{2}{3}rm +1$. Thus $n_2\geq 3$. A similar computation can now be used to determine a lower bound for $n_3$. We have
	  	\begin{eqnarray*}
	 	&& 2n_2 + 3n_3 +  4(n_{\sigma}-n_2-n_3)\leq n \implies  n_{3}\geq 4n_{\sigma}-2n_2-n \\
	 	&\implies& n_3 > 4(2mN+\frac{2}{3}rm+1) - 4m -(6mN+2rm+1). 
	 	\end{eqnarray*} 
	 	In the last inequality we have used the lower bound of $n_{\sigma}$ as given above and also the fact that $n_2<2m$. This finally yields $n_3>2mN+\frac{2}{3}rm-4m+3$. 
	 
		Observe that  when $r\in \{0,1\}$ we have $N\geq 3$ since $k>6$ and $k$ is even. So, in this case,
		$$n_3>2m+\frac{2}{3}rm+3.$$
		When $r=2$, we have $N\geq 2$. Thus $n_3>\frac{4}{3}m+3$. Thus, collectively we can conclude that $n_3\geq \frac{4}{3}m+3$.
	 
		We have $\frac{4}{3}l+1=4m+1$. Let us assume that $4m+1$ is divisible by 3. Since $n_3>\frac{4m+1}{3}$,  we can choose $\frac{4m+1}{3}$ many distinct 3-cycles from $dcd*(\sigma)$. Let $\tau$ be the product of these cycles. Then, clearly $\tau \in A_{\frac{4l}{3}+1}$ and we are done. Suppose $4m+1\equiv 1\;(\text{mod }3)$, then 3 divides $4m$. Once again since $n_3>\frac{4}{3}m$, we choose $\frac{4m}{3}$ many distinct 3-cycles from $dcd*(\sigma)$. Let $\tau$ be the product of these 3-cycles. Then $\tau\in A_{\frac{4l}{3}}$ and we are done. Finally, let us assume that $4m+1\equiv 2\;(\text{mod }3)$. In this case $4m-4$ is divisible by $3$. We already have that $n_2\geq 3$. We choose $\frac{4m-4}{3}$ many distinct 3-cycles and two distinct 2-cycles from $dcd*(\sigma)$. Let $\tau$ be the product of these cycles. Then $\tau \in A_{\frac{4l}{3}}$. This completes the proof.
	\end{proof}

	\begin{proof}[\bf{Proof of Proposition~\ref{main_decomposibility_result}}]
		By Lemma~\ref{Theorem_4.2_for_l=9} and Lemma~\ref{Theorem_4.2_for_k>6}, we can assume that $k=6$, $3\mid l$ and $l\geq 12$. In this case, we take $m=\frac{l}{3}$ where $m\geq 4$. Then $n=\frac{2}{3}kl+1=12m+1$. Let $\sigma \in A_n$.  We have $m_{\sigma}=n=12m+1$ and $n_{\sigma}>4m+1$. If $n_2\geq \frac{2l}{3}=2m$, then we can choose $\tau$ to be a product of $\frac{2l}{3}$ many distinct 2-cycles from $dcd*(\sigma)$. Clearly, $\tau \in A_{\frac{4l}{3}}$ and we are done. So we now assume that $n_2<\frac{2l}{3}=2m$. By Lemma~\ref{even_permutation_characterization}, $m_{\sigma}+n_{\sigma}$ is even and thus we can take $n_{\sigma}\geq 4m+3$. Now we follow the same procedure as in the previous lemma and obtain lower bounds for $n_2$ and $n_3$. We have 
	 	\begin{eqnarray*}
		2n_2 + 3(n_{\sigma}-n_2)\leq n &\implies&  n_{2}\geq 3n_{\sigma}-n\geq 3(4m+3)-12m-1=8
		\end{eqnarray*} 
		Further,
		\begin{eqnarray*}
		&& 2n_2 + 3n_3 +  4(n_{\sigma}-n_2-n_3)\leq n \implies n_3\geq 4n_{\sigma}-2n_2-n \\ &&
		\implies n_3 > 4(4m+3)-4m-12m-1=11
		\end{eqnarray*}
	
		If $n_3\geq \frac{4m+1}{3}$ then we can produce a $\tau$ as was done in the previous lemma (see the choice of $\tau$ in the last paragraph of the proof of Lemma~\ref{Theorem_4.2_for_k>6}). Thus we further assume that $n_3<\frac{4m+1}{3}$. We have  $m\geq 9$ since $n_3> 11$. We now estimate a lower bound for $n_2$. We have
		\begin{eqnarray*}
		&& 2n_2 + 3n_3 +  4(n_{\sigma}-n_2-n_3)\leq n \implies  n_{2}\geq \frac{4n_{\sigma}-n_3-n}{2} \\
		&\implies& n_2 > \frac{4(4m+3)-\frac{4m+1}{3}-12m-1}{2}= \frac{\frac{8m-1}{3}+11}{2} = \frac{4m+1}{3}+5.
		\end{eqnarray*} 
		In this case, we further assume that $n_3\geq \frac{4m}{9}$. Thus $\frac{4m}{9}\leq n_3< \frac{4m+1}{3}$. Now we produce a suitable $\tau$ by making some cases.
	
		\textbf{\underline{Case I}:} Suppose $4m+1\equiv 0\;(\text{mod }3)$. If $\frac{4m+1}{3}\equiv 0\;(\text{mod }3)$, then $\frac{4m+1}{9}$ is a positive integer. We can choose $\left(\frac{4m+1}{3}+3\right)$ many distinct 2-cycles and $\left(\frac{4m+1}{9} -2\right)$ many distinct 3-cycles in $dcd*(\sigma)$. Let $\tau$ be the product of these cycles. Then $\tau \in A_{4m+1}=A_{\frac{4l}{3}+1}$ since $\frac{4m+1}{3}+3$ is even.  If $\frac{4m+1}{3}\equiv 1\;(\text{mod }3)$, then $\frac{4m-2}{9}$ is a positive integer. We can choose $\left(\frac{4m+1}{3}+3\right)$ many distinct 2-cycles and $\left( \frac{4m-2}{9}-2\right)$ many distinct 3-cycles in $dcd*(\sigma)$. Let $\tau$ be the product of these cycles. Then $\tau \in A_{4m}=A_{\frac{4l}{3}}$ since $\frac{4m+1}{3}+3$ is even. Finally, let $\frac{4m+1}{3}\equiv 2\;(\text{mod }3)$, whence $\frac{4m-5}{9}$ is an integer. We can choose $\left(\frac{4m+1}{3}+5\right)$ many distinct 2-cycles and $\left(\frac{4m-5}{9} -3\right)$ many distinct  3-cycles in $dcd*(\sigma)$. Let $\tau$ be the product of these cycles. Since $m\geq 9$, $\left(\frac{4m-5}{9} -3\right)\geq \frac{4}{9}$. We can conclude that $\left(\frac{4m-5}{9} -3\right)\geq 1$. Further, $\tau \in A_{4m}=A_{\frac{4l}{3}}$ since $\frac{4m+1}{3}+5$ is even.
	
		\medskip
	
		\textbf{\underline{Case-II}:} Suppose $4m+1\equiv 1\;(\text{mod }3)$, that is, $3\mid 4m$.  If $\frac{4m}{3}\equiv 0\;(\text{mod }3)$, then $\frac{4m}{9}$ is a positive integer. We can choose $\frac{4m}{3}$ many distinct 2-cycles and $\frac{4m}{9}$ many distinct 3-cycles in $dcd*(\sigma)$. Let $\tau$ be the product of these cycles. Then $\tau \in A_{4m}=A_{\frac{4l}{3}}$ since $\frac{4m}{3}$ is even.  If $\frac{4m}{3}\equiv 1\;(\text{mod }3)$, then $\frac{4m-3}{9}$ is a positive integer. We can choose $\left(\frac{4m}{3}+2\right)$ many distinct 2-cycles and $\left( \frac{4m-3}{9}-1\right)$ many distinct 3-cycles in $dcd*(\sigma)$. Let $\tau$ be the product of these cycles. Then $\tau \in A_{4m}=A_{\frac{4l}{3}}$ since $\left(\frac{4m}{3}+2\right)$ is even. Finally, let $\frac{4m}{3}\equiv 2\;(\text{mod }3)$, whence $\frac{4m-6}{9}$ is an integer. We can choose $\left(\frac{4m}{3}+4\right)$ many distinct 2-cycles and $\left(\frac{4m-6}{9} -2\right)$ many distinct 3-cycles in $dcd*(\sigma)$. Let $\tau$ be the product of these cycles. Since $m\geq 9$, we can conclude that$\left(\frac{4m-6}{9} -2\right)\geq 1$. Further, $\tau \in A_{4m}=A_{\frac{4l}{3}}$ since $\left(\frac{4m}{3}+4\right)$ is even.
	
		\medskip
	
		\textbf{\underline{Case-III}:} Suppose $4m+1\equiv 2\;(\text{mod }3)$, that is, $3\mid 4m-1$.  If $\frac{4m-1}{3}\equiv 0\;(\text{mod }3)$, then $\frac{4m-1}{9}$ is a positive integer. We can choose $\left(\frac{4m-1}{3}+1\right)$ many distinct 2-cycles and $\left(\frac{4m-1}{9}\right)$ many distinct 3-cycles in $dcd*(\sigma)$. Let $\tau$ be the product of these cycles. Then $\tau \in A_{4m+1}=A_{\frac{4l}{3}+1}$ since $\frac{4m-1}{3}+1$ is even.  If $\frac{4m-1}{3}\equiv 1\;(\text{mod }3)$, then $\frac{4m-4}{9}$ is a positive integer. We can choose $\left(\frac{4m-1}{3}+1\right)$ many distinct 2-cycles and $\frac{4m-4}{9}$ many distinct 3-cycles in $dcd*(\sigma)$. Let $\tau$ be the product of these cycles. Then $\tau \in A_{4m}=A_{\frac{4l}{3}}$ since $\left(\frac{4m-1}{3}+1\right)$ is even. Finally, let $\frac{4m-1}{3}\equiv 2\;(\text{mod }3)$, whence $\frac{4m-7}{9}$ is an integer. We can choose $\left(\frac{4m-1}{3}+3\right)$ many distinct 2-cycles and $\left(\frac{4m-7}{9} -1\right)$ many distinct 3-cycles in $dcd*(\sigma)$. Let $\tau$ be the product of these cycles. Further, $\tau \in A_{4m}=A_{\frac{4l}{3}}$ since $\left(\frac{4m-1}{3}+3\right)$ is even.
	
		\noindent This exhausts all possible cases and thus we are done when $\frac{4m}{9}\leq n_3\leq \frac{4m+1}{3}$. Hence we assume $n_3< \frac{4m}{9}$. In this case, we estimate $n_4$ which is the number of 4-cycles in $dcd*(\sigma)$. We have,
		\begin{eqnarray*}
		&& 2n_2+3n_3+4n_4+5(n_{\sigma}-n_2-n_3-n_4)\leq n \implies n_4\geq 5n_{\sigma}-3n_2-2n_3-n \\ &&\implies n_4>5(4m+3)-6m-\frac{8m}{9}-12m-1=\frac{10m}{9}+14>m
		\end{eqnarray*}
		If $m$ is even, let $\tau$ be the product of $m$ number of distinct 4-cycles from $dcd*(\sigma)$. Then $\tau \in A_{4m}=A_{\frac{4l}{3}}$. If $m$ is odd, let $\tau$ be the product of $m-1$ many 4-cycles and two 2-cycles from $dcd*(\sigma)$. Since $m-1$ is even, we conclude that $\tau \in A_{4m}=A_{\frac{4l}{3}}$. Our proof is now complete.
	\end{proof}

	We will now prove an analogue of Theorem~\ref{upper_lower_bound_for_l_divisible_by_3_l_odd} when $l$ is even. The major step towards obtaining this is an analogue of Theorem~\ref{auxillary_theorem_hgl} for $l$ even which we now prove. The following result is required.
	
	\begin{theorem}[Theorem 7, \cite{hkl}]\label{product_of_two_cycles}
		Let $\sigma \in S_n$ and let $l_1,l_2\in \mathbb{N}$, $n\geq l_1\geq l_2\geq 2$. Then $\sigma=C_1C_2$, where $C_1,C_2$ are cycles in $S_n$ of lengths $l_1,l_2$ respectively, if and only if either $n_{\sigma}=2$, $l_1,l_2$ are the lengths of the cycles in $dcd*(\sigma)$ and $l_1+l_2=m_{\sigma}$, or the following conditions hold:
		
		\begin{enumerate}
			\item $l_1+l_2=m_{\sigma}+n_{\sigma}+2s$ for some $s\in \mathbb{Z}_{\geq 0}$, and
			\item $l_1-l_2\leq m_{\sigma}-n_{\sigma}.$
		\end{enumerate}
	\end{theorem}

	The following two lemmas will also be needed.
	
	\begin{lemma}\label{cycles_as_products_of_cycles}
		Let $l,t\in \mathbb{N}$ such that $l\geq 2$ and $t\geq 1$. Suppose that $\sigma\in S_n$ be a cycle such that $\mathfrak{l}(\sigma)=l+(t-1)(l-1)$. Then $\sigma$ can be written as a product of $t$ many $l$-cycles.
	\end{lemma}

	\begin{proof}
		For $t=1$ it is obvious. So we assume that $t\geq 2$. We can choose  $\sigma$ in the standard form as follows: $$\sigma=(1\;2\cdots\;l\; l+1\;\cdots\;2l-1\;2l\;\cdots 3l-2\;\cdots\;l+(t-2)(l-1)+1\;\cdots\;l+(t-1)(l-1) ).$$
		Now it can be easily seen that
		$$\sigma=(1\;2\;\cdots\;l)(l\;\cdots\;2l-1)(2l-1\;\cdots\;3l-2)\cdots(l+(t-2)(l-1)\;\cdots\;l+(t-1)(l-1)).$$
		Thus $\sigma$ is a product of $t$ many $l$-cycles.
	\end{proof}
	
	Recall that $P(k,l;n)$ denote the set of all elements of $A_n$ which can be written as a product of $k$ many  $l$-cycles.
	\begin{lemma}\label{lengthening_of_cycles}
		Suppose $l,k\geq 2$ and $n\geq 4$ be natural numbers such that when $k$ is even, $2\leq l \leq n-1$ and when $k$ is odd, $2\leq l\leq n-2$. Then, $P(k,l;n)\subseteq P(k,l+1;n)$ when $k$ is even and $P(k,l;n)\subseteq P(k,l+2;n)$ when $k$ is odd.
	\end{lemma}
	
	\begin{proof}
		Follows from the proof of Proposition 15 of \cite{bh}.
	\end{proof}
	The following theorem is an analogue of Theorem~\ref{auxillary_theorem_hgl} when $l$ is even.
	
	\begin{theorem}\label{auxillary_theorem_hgl_l_even}
		Let $k,l,n\in \mathbb{N}$ be such that $k\geq 2$ is even, $3\mid l$, $l$ is even and suppose that $12\leq l\leq n\leq \frac{2}{3}kl+1$. Moreover, let $\sigma\in A_n$ be such that $n_{\sigma}\leq \frac{n+2}{3}$. Then $\sigma \in P(k,l;n)$.
	\end{theorem}

	\begin{proof}
		For $k=2$ and $k=4$ this result holds by Corollary~\ref{n(2,l)_exact_value} and Corollary~\ref{n(3,l)_n(4,l)_exact_value} respectively. So we can assume $k\geq 6$. Fix $n$ such that $12\leq l\leq n\leq \frac{2}{3}kl+1$.  Let $\sigma \in A_n$ such that $n_{\sigma}\leq \frac{n+2}{3}$. Our aim is to prove that $\sigma \in P(k,l;n)$. We can assume that $\sigma \neq 1$. Lemma~\ref{increasing_cycles} will be used throughout this proof without any further mention. We make some cases based on the value of $m_{\sigma}$. Note that since $\sigma \in A_n$, $m_{\sigma}\geq 3$.
		
		\medskip
		
		\textbf{\underline{Case-I}:} Suppose that $3\leq m_{\sigma}\leq l-1$. Then $\lceil \frac{3m_{\sigma}}{8} \rceil\leq m_{\sigma}$. Using Theorem~\ref{n(4,l)_bounds}, we conclude that $\sigma \in P(4,m_{\sigma};n)$. This implies $\sigma \in P(k,m_{\sigma};n)$. Since $m_{\sigma}< l$, using Lemma~\ref{lengthening_of_cycles}, we get $\sigma\in P(k,l;n)$ as required.
		
		\medskip
		
		\textbf{\underline{Case-II}:} Suppose that $l\leq m_{\sigma}\leq 2l$. Then,  $\frac{m_{\sigma}}{2}\leq l\leq m_{\sigma}$. This implies $\lceil \frac{3m_{\sigma}}{8} \rceil\leq l\leq m_{\sigma}$. From Theorem~\ref{n(4,l)_bounds} we conclude that $\sigma \in P(4,l;n)$. Therefore, $\sigma \in P(k,l;n)$ as required.
		
		\medskip
		
		\textbf{\underline{Case-III}:}  Note that $m_{\sigma}\geq 2n_{\sigma}$. Thus, in this final case, we assume that $m_{\sigma}\geq \mathrm{max}(2l+1,2n_{\sigma})$ . Clearly,
		$$m_{\sigma}-n_{\sigma}\geq \mathrm{max}(2l+1-n_{\sigma}, n_{\sigma})\geq l+1$$
		Now we estimate $m_{\sigma}+n_{\sigma}$. We use the fact that $l\geq 12$. In fact, it is enough to use a looser inequality, namely $l\geq 9$.
		\begin{eqnarray*}
		m_{\sigma}+n_{\sigma} &\leq& n+\frac{n+2}{3}\leq \frac{8}{9}kl+2 \\
		&\leq& 2l+\frac{l}{l-1}(\frac{8}{9}k-2)(l-1) +2 \leq 2l+\frac{9}{8}(\frac{8}{9}k-2)(l-1)+2\\
		&\leq& 2l+(k-2)(l-1)-\frac{l-9}{4}\leq 2l+(k-2)(l-1).
		\end{eqnarray*}
		Let $s\in \mathbb{N}$ be minimal such that
		\begin{equation}\label{minimal}
		m_{\sigma}+n_{\sigma}\leq 2l+(s-2)(l-1).
		\end{equation}
		We have $3\leq s\leq k$. Suppose $s=3$. Then, $m_\sigma+n_\sigma\leq 3l-1$. By Lemma~\ref{even_permutation_characterization}, $m_{\sigma}+n_{\sigma}\leq 3l-2=(2l-1)+(l-1)$. Suppose $l_1=2l-1$ and $l_2=l-1$. Since $m_{\sigma}+n_{\sigma}\leq l_1+l_2$ and $m_{\sigma}-n_{\sigma}\geq l_1-l_2=l$, using Theorem~\ref{product_of_two_cycles} we get that $\sigma=C_1C_2$ where $C_1$ and $C_2$ are cycles of length $2l-1$ and $l-1$ respectively. Clearly $C_1,C_2\in P(2,l;n)$ and hence $\sigma\in P(4,l;n)$. Thus $\sigma\in P(k,l;n)$.  Therefore, we can assume that $s\geq 4$. From \eqref{minimal} we get,
		$$m_{\sigma}+n_{\sigma}\leq [l+(\lceil s/2 \rceil-1)(l-1)]+[l+(\lfloor s/2 \rfloor-1)(l-1)]=l_1+l_2.$$
		By minimality of $s$ we have
		$$2l+(s-3)(l-1)<m_{\sigma}+n_{\sigma}\leq \frac{4n+2}{3}.$$
		Now
		\begin{eqnarray*}
			&& \frac{4n+2}{3}>2l+(s-3)(l-1) \implies \frac{4n}{3}>2l+(s-3)(l-1)-\frac{2}{3}\\ &&
			\implies n>\frac{3}{2}l+\frac{3}{4}(s-3)(l-1)-\frac{1}{2}
		\end{eqnarray*}
		Now assume that $l_1>n$. Then from the previous inequality we have
		\begin{eqnarray*}
		&& l+(\lceil s/2 \rceil-1)(l-1)>\frac{3}{2}l+\frac{3}{4}(s-3)(l-1)-\frac{1}{2} \\ &&
		\implies (\lceil s/2 \rceil-1)(l-1)-\frac{1}{2}l+\frac{1}{2}-\frac{3}{4}(s-3)(l-1)>0 \\ &&
		\implies (\lceil s/2 \rceil-\frac{3}{4}s)(l-1)+(\frac{5}{4}-\frac{1}{2})(l-1)>0 \\ &&
		 \implies (\lceil s/2 \rceil-\frac{3}{4}s+\frac{3}{4})(l-1)>0
		\end{eqnarray*}
		This implies that $\lceil s/2 \rceil-\frac{3}{4}s+\frac{3}{4}>0$. Suppose $s=2t$ for some $t\geq 1$. Then the above inequality yields $t<\frac{3}{2}$ which implies $t=1$, whence $s=2$, a contradiction to the choice of $s$. If $s=2t+1$ for some $t\geq 1$, then the above inequality yields $t<2$, once again a contradiction to our choice of $s$. Therefore, we conclude that $n\geq l_1\geq l_2$.
	
		\medskip
	
		Now assume that $s$ is even. Let $s=2t$ where $t\geq 2$. Then $l_1=l_2=l+(t-1)(l-1)$. Since $0=l_1-l_2\leq m_{\sigma}-n_{\sigma}$, it follows by Theorem~\ref{product_of_two_cycles} that $\sigma=C_1C_2$ where $C_1$ and $C_2$ are both cycles of length $l+(t-1)(l-1)$. By Lemma~\ref{cycles_as_products_of_cycles}, $C_1,C_2\in P(t,l;n)$ which implies that $\sigma \in P(2t,l;n)=P(s,l;n)$. We have $s\leq k$ and $s$ is even which yields that $\sigma \in P(k,l;n)$. Finally assume that $s$ is odd, say $s=2t+1$ where $t\geq 2$. In this case $l_1=l+t(l-1)$ and $l_2=l+(t-1)(l-1)$. Since $l_1\geq l_2$,  we conclude that $m_{\sigma}+n_{\sigma}\leq l_1+l_2\leq 2l_1$. Using Theorem~\ref{product_of_two_cycles}, we conclude that $\sigma=C_1C_2$ where both $C_1$ and $C_2$ are cycles of length $l+t(l-1)$. By Lemma~\ref{cycles_as_products_of_cycles}, we get $C_1,C_2\in P(t+1,l;n)$. Thus $\sigma\in P(2t+2,l;n)=P(s+1,l;n)$. Since $s$ is odd, $s\leq k-1$ which implies $s+1\leq k$. Since $s+1$ is even, it follows that $\sigma\in P(k,l;n)$ thereby completing the proof.
	\end{proof}
	
	The following proposition is required to proceed further.
	
	\begin{proposition}[Corollary 2.2, \cite{hgl}]\label{secondary_decomposability_lemma}
		Let $n$ be an even integer and let $\sigma \in S_n$, satisfying $n_{\sigma}>\frac{n+1}{3}$. Then for each even integer $m$ satisfying $2\leq m\leq n-2$, there exists non-trivial disjoint permutations $\rho$ and $\phi$ of $S_n$ such that $\sigma=\rho\phi$, $|\supp(\rho)|\leq m$ and $|\supp(\phi)|\leq n-m$.
	\end{proposition}
	
	We finish this section with the following analogue of  Theorem~\ref{upper_lower_bound_for_l_divisible_by_3_l_odd} in the case when $l$ is even. The proof is similar as in Theorem~\ref{upper_lower_bound_for_l_divisible_by_3_l_odd} but we write it for completeness.
	
	\begin{theorem}\label{upper_lower_bounds_n(k,l)_l_even}
		Suppose $k,l$ are natural numbers such that $k\geq 2$, $l\geq 12$, $3\mid l$ and $l$ is even. Then $\frac{2}{3}kl\leq n(k,l)\leq \frac{2}{3}kl+1$.
	\end{theorem}
	
	\begin{proof}
		The fact that $n(k,l)\leq \frac{2}{3}kl+1$ follows from Theorem~\ref{general_upper_bound_n(k,l)}. Now we show that $n(k,l)\geq \frac{2}{3}kl$. We use induction on $k$. When $k=2$ the result follows from Theorem~\ref{bounds_n(2,l)} and when $k=4$ it follows from Theorem~\ref{n(4,l)_bounds}. Let $M\geq 4$. Suppose that the result holds for all $k$ such that $4\leq k\leq M$ and $k$ even. We show that the result holds for $k=M+2$. Let $n'=\frac{2}{3}(M+2)l$ and suppose $\sigma \in A_{n'}$. If $n_{\sigma}\leq\frac{n'+2}{3}$ then by Theorem~\ref{auxillary_theorem_hgl_l_even}, we conclude that $\sigma \in P(M+2,l;n')$. Now assume that $n_{\sigma}>\frac{n'+2}{3}$. Since $n'$ is even, taking $m=\frac{4}{3}l$ in Proposition~\ref{secondary_decomposability_lemma}, we conclude that $\sigma=\rho \phi$ where $\rho, \phi$ are disjoint permutations on $\frac{2}{3}Ml$ and $\frac{4l}{3}$ letters respectively. Two cases can arise namely $\rho$ and $\phi$ both are even permutations or, both are odd permutations. Firstly assume that $\rho\in A_{\frac{2}{3}Ml}$ and $\phi\in A_{\frac{4}{3}l}$. By induction hypothesis, $\rho \in P(M,l;\frac{2}{3}Ml)$ and $\phi\in P(2,l;\frac{4}{3}l)$. Thus $\sigma \in P(M+2,l;n')$ and we are done in this case.
		
		\medskip
		
		Now let us assume that both $\rho$ and $\phi$ are odd permutations and that $n_{\rho}\leq \frac{\frac{2}{3}Ml+2}{3}$. Let $\tau=(u\;v)$ be a transposition where $u\in \supp(\rho)$ and $v\in \supp(\phi)$. Take $\rho^*=\rho\tau$ and $\phi^*=\tau \phi$. Thus $\sigma=\rho^*\phi^*$. It is immediate that $n_{\rho^*}=n_{\rho}\leq \frac{\frac{2}{3}Ml+2}{3}$. Further $\rho^*$ and $\phi^*$ are even permutations in $\frac{2}{3}Ml+1$ and $\frac{4}{3}l+1$ letters respectively. From Theorem~\ref{auxillary_theorem_hgl_l_even} and Corollary~\ref{n(2,l)_exact_value}, $\rho^*\in P(M,l;\frac{2}{3}Ml+1)$ and $\phi^*\in P(2,l;\frac{4}{3}l+1)$, whence $\sigma \in P(M+2,l;n')$ and we are done. 
		
		Now assume that $n_{\rho}> \frac{\frac{2}{3}Ml+2}{3}$. Once again using Proposition~\ref{secondary_decomposability_lemma}, we have $\rho=\rho_1\rho_2$ where $\rho_1$ and $\rho_2$ disjoint permutations on $\frac{2}{3}(M-2)l$ and $\frac{4}{3}l$ letters respectively. Furthermore, $\rho_1$ and $\rho_2$ have opposite parity. If $\rho_1$ is an odd permutation, it follows that $\rho_1\phi$ is an even permutation on $\frac{2}{3}Ml$ letters. By induction hypothesis, $\rho_1\phi \in P(M,l;\frac{2}{3}Ml)$ and since $\rho_2$ is even, by Corollary~\ref{n(2,l)_exact_value}, $\rho_2\in P(2,l;\frac{4}{3}l)$. Since $\sigma=\rho_1\rho_2\phi$, it follows that $\sigma \in P(M+2,l;n')$. If instead we assume that $\rho_2$ is an odd permutation then $\rho_2\phi$ is an even permutation in $\frac{8}{3}l$ letters. By Corollary~\ref{n(3,l)_n(4,l)_exact_value}, we get that $\rho_2\phi\in P(4,l; \frac{8}{3}l)$. Since $\rho_1$ is an even permutation on $\frac{2}{3}(M-2)l$ letters, it follows by induction hypothesis that $\rho_1 \in P(M-2,l;\frac{2}{3}(M-2)l)$. Since $\sigma = \rho_1\rho_2\phi$, $\sigma \in P(M+2,l;n')$. This completes the induction process and we get our result.
	\end{proof}
	
	\section{Proof of Theorem~\ref{n(k,l)_exact_value_k_even}}
	
	Notice that in Theorem~\ref{auxillary_theorem_hgl_l_even} and Theorem~\ref{upper_lower_bounds_n(k,l)_l_even}, we have assumed $l$ is even, $3\mid l$ and $l\geq 12$. As a final ingredient towards the proof of the main theorem, we prove the required results for $l=6$. Note that when $l=6$, we have $\frac{2}{3}kl=4k$.
	
	\begin{proposition}\label{n(k,6)_lemma1}
		Let $k\geq 2$ be even and assume that $6\leq n\leq 4k+1$. Let $\sigma \in A_n$ be such that $n_{\sigma}\leq k+1$. Then, $\sigma\in P(k,6;n)$.
	\end{proposition}
	
	\begin{proof}
		For $k=2,4$ this result holds by Corollary~\ref{n(2,l)_exact_value} and Corollary~\ref{n(3,l)_n(4,l)_exact_value}. So we can assume that $k\geq 6$. Fix $n$ such that $6\leq n\leq 4k+1$.  Let $\sigma \in A_n$ such that $n_{\sigma}\leq k+1$. Note that $m_{\sigma}\leq 4k+1$.
		We have,
		$$m_{\sigma}+n_{\sigma}\leq 4k+1+k+1=5k+2=\frac{5k}{2}+1+\frac{5k}{2}+1.$$
		Take $l_1=l_2=\frac{5k}{2}+1$. Clearly $l_1,l_2\in \mathbb{N}$ since $k$ is even. Further $m_{\sigma}-n_{\sigma}\geq 0=l_1-l_2$. Applying Theorem~\ref{product_of_two_cycles}, we conclude that $\sigma=C_1C_2$ where $C_1$ and $C_2$ are cycles each of length $\frac{5k}{2}+1$. Observe that $\frac{5k}{2}+1=6+(\frac{k}{2}-1)5$. By Lemma~\ref{cycles_as_products_of_cycles}, both $C_1$ and $C_2$ can be written as a product of $\frac{k}{2}$ many 6-cycles. Thus $\sigma\in P(k,6;n)$ and we are done. 
	\end{proof}

	\begin{proposition}\label{n{k,6}_lemma2}
		Let $k\geq 4$ be even. Let $n=4k$. Suppose $\sigma \in A_n$ be such that $n_{\sigma}\geq k+2$. Then $\sigma$ can be written as a product of two disjoint permutations $\phi$ and $\tau$ such that $\phi \in A_{4(k-2)-\epsilon}$ and $\tau \in A_{8+\epsilon}$ where $\epsilon \in \{0,1\}$.
	\end{proposition}

	\begin{proof}
	 	Let $\sigma \in A_n$ such that $\sigma$ satisfies the hypothesis of the above statement. We will write $n_i$ to mean $n_i(\sigma)$. If $n_2\geq 4$, we take $\tau$ to be a product of four distinct 2-cycles from $dcd*(\sigma)$. Then $\tau \in A_8$ and we are done. Therefore, we assume $n_2\leq 3$. Now we estimate the number of $3$-cycles $n_3$. We have,
	 	\begin{equation}\label{n_3bound}
	 	2n_2+3n_3+4(n_{\sigma}-n_2-n_3)\leq m_{\sigma} \implies n_3\geq 4n_{\sigma}-2n_2-m_{\sigma}.
	 	\end{equation}
	 	Now we consider two cases. Suppose that $m_{\sigma}<4k$. Then, $n_3\geq 4(k+2)-6-(4k-1)=3$. In this case, we choose $\tau$ to be a product of three distinct 3-cycles from $dcd*(\sigma)$. Thus $\tau \in A_9$ and we are done. Finally assume that $m_{\sigma}=4k$. In this case from \eqref{n_3bound}, if we assume that $n_2\leq 2$, we get $n_3\geq 3$, whence we get our desired $\tau$. So assume that $n_2=3$. From \eqref{n_3bound}, we get $n_3\geq 2$. Once again by the same argument as before, we assume that $n_3=2$. Now we proceed to give a lower bound for $n_4$. We have
	 	\begin{eqnarray*}
	 	&& 2n_2+3n_3+4n_4+5(n_{\sigma}-n_2-n_3-n_4)\leq 4k \\ && \implies n_4\geq 5n_{\sigma}-2n_3-3n_2-4k=5(k+2)-13-4k=k-3.
	 	\end{eqnarray*}
 		Since $k\geq 4$ we conclude that $n_4\geq 1$. Now choose $\tau$ to be  the product of a 2-cycle, a 3-cycle and a 4-cycle from $dcd*(\sigma)$. Clearly $\tau \in A_9$ and the proof is complete.
	\end{proof}

	\begin{theorem}\label{lowerbound_n(k,6)}
		Let $k\geq 2$ be even. Then $n(k,6)\geq 4k$.
	\end{theorem}

	\begin{proof}
		We use induction on $k$. If $k=2$, the result follows from Corollary~\ref{n(2,l)_exact_value} and when $k=4$, it follows from Corollary~\ref{n(3,l)_n(4,l)_exact_value}. Let us now assume that the result holds for every $k$ such that $4\leq k\leq M$ where both $k$ and $M$ are even. Let $\sigma \in A_{4(M+2)}=A_{4M+8}$. If $n_{\sigma}\leq M+3$, then by Proposition~\ref{n(k,6)_lemma1} (putting $k=M+2$) we conclude that $\sigma \in P(M+2,6,4M+8)$. Now we assume that $n_{\sigma}\geq M+4$. In this case, by the previous lemma (putting $k=M+2$) $\sigma$ can be written as $\sigma=\phi \tau$ where $\phi\in A_{4M-\epsilon}$ and $\tau\in A_{8+\epsilon}$ where $\epsilon \in \{0,1\}$. By induction hypothesis, $\phi \in P(M,6;4M-\epsilon)$ and by Corollary~\ref{n(2,l)_exact_value}, $\tau \in P(2,6;8+\epsilon)$. We conclude that $\sigma \in P(M+2,6;4M+8)$. This completes the induction and we conclude that $n(k,6)\geq 4k$ for every even number $k\geq 2$.
	\end{proof}

	Now we have all the ingredients to prove Theorem~\ref{n(k,l)_exact_value_k_even}.
	
	\begin{proof}[Proof of Theorem~\ref{n(k,l)_exact_value_k_even}] 
		Let $l\in \mathbb{N}$, $3\mid l$ and $l\geq 9$. Assume further that $k$ is even. Let $n=\frac{2}{3}kl+1$. Suppose $\sigma \in A_{n}$ with $m_{\sigma}<n$. Then, $\sigma$ can be considered in $A_{n-1}$. From Theorem~\ref{upper_lower_bound_for_l_divisible_by_3_l_odd} and Theorem~\ref{upper_lower_bounds_n(k,l)_l_even}, we conclude that $\sigma \in P(k,l;n-1)$, whence it follows that $\sigma\in P(k,l;n)$.
		
		Thus it remains to prove that if $m_{\sigma}=n$, then $\sigma \in P(k,l;n)$. Let $\Gamma_n=\{\sigma\in A_n \mid m_{\sigma}=n\}$. We claim that $\Gamma_n \subseteq P(k,l;n)$. We prove this by induction on $k$. When $k=2$, the result follows from Corollary~\ref{n(2,l)_exact_value} and when $k=4$ it follows from Corollary~\ref{n(3,l)_n(4,l)_exact_value}. Let $M\geq 4$ be even. Suppose that the result holds for every $k$ such that $4\leq k\leq M$. To complete the induction, we now prove that the result holds for $k=M+2$. Let $\sigma \in \Gamma_{n'}$ where $n'=\frac{2}{3}(M+2)l+1$. If $n_{\sigma}\leq \frac{n'+2}{3}$, then by Theorem~\ref{auxillary_theorem_hgl} and Theorem~\ref{auxillary_theorem_hgl_l_even} we conclude that $\sigma \in P(M+2,l;n')$. Now we assume that $n_{\sigma}> \frac{n'+2}{3}$. Using Proposition~\ref{main_decomposibility_result}, we can write $\sigma$ as a disjoint product of $\phi$ and $\tau$ where $\phi \in \Gamma_{\frac{2}{3}Ml+\epsilon}$ and $\tau \in \Gamma_{\frac{4}{3}l+1-\epsilon}$ where $\epsilon \in \{0,1\}$. By induction hypothesis we conclude that $\Gamma_{\frac{2}{3}Ml+\epsilon}\subseteq P(M,l,\frac{2}{3}Ml+\epsilon)$, and by Corollary~\ref{n(2,l)_exact_value} we get that  $\Gamma_{\frac{4}{3}l+1-\epsilon} \subseteq P(2,l;\frac{4}{3}l+1-\epsilon)$. This proves that $\Gamma_{n'}\subseteq P(M+2,l;n')$. This implies that $n(k,l)\geq \frac{2}{3}kl+1$. The fact that $n(k,l)\leq \frac{2}{3}kl+1$ follows from Theorem~\ref{upper_lower_bound_for_l_divisible_by_3_l_odd}.
		
		\medskip

		Finally we establish the theorem for $l=6$. In this case, we have $\frac{2}{3}kl=4k$. Note that by Theorem~\ref{general_upper_bound_n(k,l)} we have $n(k,6)\leq 4k+1$. Let $n=4k+1$. Suppose $\sigma \in A_{n}$ with $m_{\sigma}<n$. Then, $\sigma$ can be considered in $A_{n-1}$. By Theorem~\ref{lowerbound_n(k,6)}, we conclude that $\sigma \in P(k,6,n-1)$, whence it follows that $\sigma\in P(k,6;n)$. Now suppose that $m_{\sigma}=n=4k+1$. We use induction on $k$. When $k=2$, the result follows from Corollary~\ref{n(2,l)_exact_value} and when $k=4$ it follows from Corollary~\ref{n(3,l)_n(4,l)_exact_value}. Let $M\geq 4$ be even. Suppose that the result holds for every even $k$ such that $4\leq k\leq M$. To complete the induction, we now prove that the result holds for $k=M+2$. Take $n'=4(M+2)+1=4M+9$. Suppose $n_{\sigma}\leq M+3$. By Proposition~\ref{n(k,6)_lemma1} (putting $k=M+2$), we conclude that $\sigma \in P(M+2,6,n')$. Thus we can  further assume that $n_{\sigma}\geq M+4$. Along similar lines as in Proposition~\ref{n{k,6}_lemma2}, it follows that $\sigma$ can be written as a product of two disjoint permutations $\phi$ and $\tau$ such that $\phi \in A_{4M+\epsilon}$ and $\tau \in A_{8+(1-\epsilon)}$ where $\epsilon \in \{0,1\}$. By induction hypothesis, $\phi \in P(M,6;4M+\epsilon)$ and by Corollary~\ref{n(2,l)_exact_value} we get that $\tau \in P(2,6,8+(1-\epsilon))$. We conclude that $\sigma \in P(M+2,6; n')$. The induction is complete and the proof follows.
	\end{proof}

	\subsection*{Acknowledgment} The first named author thanks the Ministry of Education (Government of India) for the Prime Minister's Research Fellowship. The second and third named authors would like to acknowledge the support of IISER Mohali institute post-doctoral fellowship during this work. We thank Dr. Gurleen Kaur for helpful discussions during this work. We also thank Prof. Amit Kulshrestha for his interest in this work. The third named author thanks Prof. Chetan Balwe for his support.
	\bibliographystyle{amsalpha}
	\bibliography{References}
\end{document}